\documentclass[11pt,lefteqn]{article}

\usepackage{amsmath,amsfonts,amsthm,amssymb}

\newcommand{\C}{\mathbb C}

\newcommand{\Aa}{\mathcal A_R}
\newcommand{\aR}{\partial a_R}

\newcommand*{\clos}[1]{\overline{#1}}

\newtheorem{theorem}{Theorem}
\newtheorem{lemma}[theorem]{Lemma}

\begin{document}
\date {\today}
\title{Spectral estimates in the quantum annulus  and in the numerical annulus}

\author{Michel Crouzeix, \footnote{Univ.\,Rennes, CNRS, IRMAR\,-\,UMR\,6625, F-35000 Rennes, France, email: michel.crouzeix@univ-rennes.fr}}

\maketitle

\begin{abstract} 
We improve previous estimates for matrices belonging to the quantum annulus or to the numerical annulus.
\end{abstract}

\paragraph{2000 Mathematical subject classifications\,:}47A25 ; 47A30

\noindent{\bf Keywords\,:}{ numerical range, spectral set}

\section{Introduction}
We consider the annulus $\Aa=\{z\in\C\,;R^{-1}<|z|<R\}$ with $R>1$ and a bounded linear operator $A$ on a Hilbert space. 
Recall that this annulus is called a $K$-spectral set for the operator $A$ if  the  spectrum of $A$ is contained in $\Aa$ and if 
\[
\|f(A)\|\leq K,\quad\textrm{for all rational functions $f$ bounded by $1$ in }\Aa.
\]
We denote by $K(A,R)$ the smallest of such $K$ and by $K(R)$ the supremum of these $K(A,R)$. 
The aim of this article is to show that for matrices $A\in\C^{d,d}$ we have
\begin{align}
K(A,R)&\leq 1{+}\sqrt{1{+}\frac{2}{R^2{+}1}},\quad\textrm{if $A$ satisfies $\|A\|<R$ and } \|A^{-1}\|<R,
\label{eq1}\\
K(A,R)&\leq 2{+}\sqrt{4{+}\frac{2}{R^2{+}1}},\quad\textrm{if $A$ satisfies $w(A)<R$ and } w(A^{-1})<R.\label{eq2}
\end{align}
The appellation quantum annulus of type $R$ indicates the class of operators on a Hilbert space satisfying $\|A\|<R$ and $\|A^{-1}\|<R$, 
while the name numerical annulus corresponds to operators
satisfying $w(A)<R$ and $w(A^{-1})<R$.
Here, we use the usual norm $\| \cdot\|$ on the Euclidean space $\C^d$ associated to the inner product 
$\langle \cdot,\cdot\rangle$, and for a matrix we use the operator norm $\|A\|=\max\{\|Au\|\,; u\in\C^d,\|u\|\leq 1\}$, and 
we denote by $w(A)=\max\{|\langle Au,u\rangle|\,; u\in\C^d,\|u\|\leq 1\}$ the numerical radius of $A$.

In the quantum case, an older (1974) estimate $K(R)\leq 2+\sqrt{\frac{R^2+1}{R^2-1}}$ is due to A.~Shields \cite{shields};
it has been improved first by V.~Paulsen and D.~Sinh \cite{PauSin} and subsequently in \cite{crzx2}; the first estimate bounded 
for $R$ close to 1, $K(R)\leq 2+\frac{R+1}{\sqrt{R^2+1}}$, has been obtained with C.~Badea and B.~Beckermann \cite{babecr} and then $K(R)\leq 1{+}\sqrt2$ with A.~Greenbaum \cite{crgr}.
Recently, J.E.~Pascoe \cite{pas} has obtained a bound close to 2 as $R$ tends to infinity.
G.~Tsikalas \cite{tsik} has shown that a bound valid for all these operators has to satisfy $K(R)\geq 2$.
In the annulus case, the work of Shields provides an estimate $K(R)\leq 4+2\sqrt{\frac{R^2+1}{R^2-1}}$;
some improvements have been given in \cite{crzx2}, a uniform bound $3{+}\sqrt{10}$ in \cite{crgr}; the latter bound was improved to $2{+}\sqrt{7}$ by M.~Jury and G. Tsikalas \cite{jutsi}. 

Our estimates \eqref{eq1}, \eqref{eq2} improve all the previous estimates; however the
previous estimates for $K(R)$ are valid for linear operators on any Hilbert spaces while
this article is limited to the finite dimension case. We do not know if our new estimates are valid in general, nor if they are valid in completely bounded form.

The organization of this paper is the following. In Section~2, we introduce the different tools used for getting the estimates: Cauchy transform of the conjugate, double layer potential, and the technique or extremal pairs (initiated in \cite{cgl}; see also \cite{sv,mmlro}) which requires a finite dimension context.
In Section~3, we provide estimates for the Cauchy transform of the conjugate in our annulus context and in Section~4 bounds 
for the double layer kernel.

\noindent {\bf Remark.}{ \it In order to prove \eqref{eq1} and \eqref{eq2} we may assume that the eigenvalues of $A$ are all distinct. Indeed, $d\times d$ 
matrices belonging to the quantum (resp. numerical) annulus with distinct eigenvalues are dense in the set of $d\times d$ matrices belonging to the quantum 
(resp. numerical) annulus.}

\section{Formalism}
We consider a rational function $f$, bounded in the annulus, and a square matrix $A$ with eigenvalues in $\Aa$. 
Then, the matrix $f(A)$ is well defined and we have the Cauchy formulae
\[
f(z)=\frac1{2\pi i}\int_{\partial \Aa}f(\sigma)\frac{d\sigma }{\sigma -z}, \quad f(A)=\frac1{2\pi i}\int_{\partial \Aa}f(\sigma)(\sigma I{-}A)^{-1}d\sigma .
\]
We will also use the Cauchy transforms of the complex conjugates of $f$,
\begin{equation}\label{eq3}
g(z):=C(\clos{f},z):=\frac1{2\pi i}\int_{\partial \Aa}\clos{f(\sigma)}\frac{d\sigma }{\sigma -z}, \quad g(A):=\frac1{2\pi i}\int_{\partial \Aa}\clos{f(\sigma)}(\sigma I{-}A)^{-1}d\sigma.
\end{equation}
In the previous formula, $\partial \Aa$ denotes the boundary of the annulus, which is the union of the two circles 
$\Gamma _R=\{R\,e^{i\theta }\,: \theta \in[0,2\pi ]\}$ and $\Gamma _r=\{r\,e^{-i\theta }\,: \theta \in[0,2\pi ]\}$; from now on, 
we use the notation $r=1/R$. We will use the following two representations for the points $ \sigma $ on this boundary:

 $\bullet\quad $on $\Gamma _R$,\quad $\sigma (s)=R\,e^{is/R}=R\,e^{i\theta }$ with $\theta =s/R$, $s\in[0,2\pi R])$,\smallskip

$\bullet\quad $on $\Gamma _r$,\quad $\sigma (s)=r\,e^{-is/r}=r\,e^{-i\theta }$ with $\theta =2\pi {+}s/r$, $s\in[-2\pi r,0)$.\smallskip

So, the boundary is parametrized either by the curvilinear abscissa $s\in \partial a_R:=[-2\pi r,2\pi R)$
or by the corresponding $\theta $. Then, for all functions $\varphi $, we will have
\[
\int_{\partial a_R}\varphi (\sigma (s))\,ds=\int_{0}^{2\pi }\varphi (R\,e^{i\theta })\,R\,d\theta +
\int_{0}^{2\pi }\varphi (r\,e^{-i\theta })\,r\,d\theta.
\]
Note that we have a counter clockwise orientation on $\Gamma_R$ and a clockwise on $\Gamma _r$. 
We introduce the transforms of $f$ by the double layer potential kernel,
\begin{equation}\label{eq4}
S(f,z):=\int_{\partial a_R}f(\sigma(s))\mu (\sigma(s),z)\,ds ,\quad S=S(f,A):=\int_{\partial a_R}f(\sigma (s))\mu (\sigma (s),A)\,ds.
\end{equation}
Here, $\mu $ denotes the kernel given by 
\begin{align}\label{eq5}
\mu (\sigma(s),z):=\frac1\pi \frac{d\arg(\sigma (s){-}z)}{ds}=\frac{1}{2\pi i}\Big(\frac{\sigma '(s)}{\sigma(s) -z}-\frac{\clos{\sigma '(s)}}{\clos{\sigma(s)} -\bar z}\Big),\\
\label{eq6}\mu (\sigma(s) ,A):=
\frac{1}{2\pi i}\big(\sigma '(s)(\sigma(s) I{-}A)^{-1}-\clos{\sigma '(s)}(\clos{\sigma(s)} I {-}A^*)^{-1}\big).
\end{align}
Note that $\mu ( \sigma (s) , z)$ is real-valued and $\mu ( \sigma (s), A)$ is self-adjoint. 
From these definitions, it is clear that (for $z\in\Aa$)
\begin{equation}\label{eq7}
f(z)+\clos{g(z)}=S(f,z)\quad\text{and}\quad S(f,A)=f(A)+g(A)^*.
\end{equation}
Note also that, if we choose the constant function $f=1$, then $g=1$, $f(A)=g(A)=I$,
\begin{equation}
\int_{\partial a_R}\mu (\sigma(s) ,z)\,ds=S(1,z)=2,\quad\text{if  }z\in\Omega\quad\text{and}\quad
\int_{\partial a_R}\mu (\sigma(s) ,A)\,ds=S(1,A)=2I. \label{S1zA}
\end{equation}

We have defined $K(A,R)$ as the smallest constant such that
\begin{equation}\label{eq9}
\|f(A)\|\leq K(A,R)\max\{|f(z)|\,;z\in\Aa\}|,
\end{equation}
for all rational functions. From Mergelyan's theorem, this inequality is also valid for all holomorphic functions in $\Aa$,  continuous up to the boundary. 
We assume that all eigenvalues of the matrix $A$ are interior to the annulus; 
so, from the compactness property of bounded holomorphic functions 
in $\Aa$, it is easily seen that there exists a function $f_0$ holomorphic in $\Aa$ satisfying $\|f_0(A)\|=K(A,R)$ and 
$\|f_0\|_{L^\infty(\Aa)}:=\sup\{|f_0(z)|,\,z\in\Aa\}=1$. Now, if $h$ is another holomorphic function whose   values and derivatives satisfy 
for all eigenvalues $\lambda _i$ of $A$ of multiplicity $r_i$
\[
h^{(j)}(\lambda _i)=f_0^{(j)}(\lambda _i),\quad \textrm{for }0\leq j\leq r_i{-}1,
\]
then, we have $h(A)=f_0(A)$; thus, from \eqref{eq9} we deduce $\|f_0\|_{L^\infty(\Aa)}\leq \|h\|_{L^\infty(\Aa)}$. Therefore, $f_0$ 
is the unique holomorphic function with minimal norm for the
values $f_0^{(j)}(\lambda _i)$ given. This implies that $f_0$ is an inner function of order $\leq d$;
thus, in particular, that $f_0\in C^\infty(\overline\Aa)$ (see \cite{gara, fish} for the distinct eigenvalues case, but I am convinced 
that this is also true in the general case \cite{crzx0}).

Now, we assume that $K(A,R)>1$ and we consider a unit vector $x_0$ such that $\|f_0(A)x_0\|=K(A,R)$; then, it is known from \cite{hol} that
$\langle f_0(A)x_0,x_0\rangle =0$. 
We set
\begin{align*}
g_0(z) = \frac{1}{2 \pi i} \int_{\partial {\mathcal A}_R} \overline{f_0( \sigma )} \frac{d \sigma}{\sigma - z}.
\end{align*}
Then, using \eqref{eq7}, we have for all $c_1$ and $c_2\in\C$,
\begin{align*}
\| f_0(A)x_0\|^2&=\langle f_0(A)x_0,f_0(A)x_0\rangle =
\langle f_0(A)x_0,S(f_0,A)x_0\rangle -\langle f_0(A)x_0,(g_0(A))^*x_0\rangle\\
&=\langle f_0(A)x_0,(S(f_0,A){-}c_1I)x_0\rangle -\langle f_0(A)(g_0(A){-}c_2I)x_0,x_0\rangle.
\end{align*}
Furthermore, \cite[Theorem\,4.5]{BG+}, there exists a probability measure $\mu $ such that 
\[
\langle f_0(A)(g_0(A){-}c_2I)x_0,x_0\rangle=\int_{\partial \Aa}(f_0(g_0{-}c_2))(\sigma )\,d_\mu \sigma;
\]
thus, we deduce
\[
\|f_0(A)\|^2\leq \|S(f_0,A){-}c_1I\|\,\|f_0(A)\|+\|g_0{-}c_2\|_{L^\infty(\Aa)},
\]
which shows that for all $c_1,\,c_2\in\C$
\begin{align}\label{eq10}
 &K(A,R)\leq \max\big(1,a+\sqrt{a^2+b}\big),\quad \textrm{with}\\[4pt]
 &a=\tfrac12\|S(f_0,A)-c_1\,I\| \textrm{   and  }b=\|g_0{-}c_2\|_{L^\infty(\Aa)}.\nonumber
\end{align}

\section{Estimate for the Cauchy transform of the conjugate of $f$.}
We define 
\[
\gamma (f)=\frac{1}{2\pi }\frac{R^2{-}1}{R^2{+}1}\int_{0}^{2\pi }\,f(e^{i\theta })\,d\theta.
\]
 Then, we have
\begin{lemma}\label{lem8}
For any rational function $f$ bounded by $1$ in ${\mathcal A}_R$, the associated function $g(z)$ defined by
\begin{equation}\label{defn_of_g}
g(z) = \frac{1}{2 \pi i} \int_{\partial {\mathcal A}_R} \overline{f( \sigma )} \frac{d \sigma}{\sigma - z} ,~~
z \in {\mathcal A}_R ,
\end{equation}
satisfies $| g(z) | \leq 1$ in ${\mathcal A}_R$  and $|g(z)-\overline{\gamma (f)}|\leq \frac{2}{R^2{+}1}$.
\end{lemma}
\begin{proof}
Recall that $g$ has a continuous extension to the boundary given for $\sigma_0 \in \partial {\mathcal A}_R$ by
\[
g( \sigma_0 ) = \int_{\partial {\mathcal A}_R} \overline{f( \sigma(s) )} \mu( \sigma(s) , \sigma_0 )\,ds ,~~
\mbox{ with } \mu ( \sigma (s) , z ) = \frac{1}{2 \pi i} \left( \frac{\sigma' (s)}{\sigma (s) -z} -
\frac{\overline{\sigma' (s)}}{\overline{\sigma(s)} - \bar{z}} \right) ,
\]
where $s$ denotes arclength on $\partial {\mathcal A}_R$.  Let $f_{\theta} (z) = f( z e^{i \theta} )$,
$g_{\theta} (z) = g( z e^{i \theta} )$, $\tilde{f} (z) = f(1/z)$, and $\tilde{g} (z) = g(1/z)$.
Then, it is easily verified that if we replace $f$ by $f_{\theta}$ (resp.~by $\tilde{f}$), the 
associated function $g$ in (\ref{defn_of_g}) is replaced by $g_{\theta}$ (resp.~by $\tilde{g}$).
From this and the maximum principle, it suffices to show that $| f | \leq 1$ in ${\mathcal A}_R$
implies $| g(r) | \leq 1$.
Note that
\[
\mu ( \sigma , r) = - \frac{R}{2 \pi} ,~~\mbox{ if } \sigma = r e^{-i \theta} ,
\]
\[
\mu ( \sigma , r ) = \frac{R}{\pi} \frac{R^2 - \cos \theta}{R^4 - 2 R^2 \cos \theta + 1} ,~~\mbox{ if } 
\sigma = R e^{i \theta} .
\]
On $\Gamma_r$, we write $\sigma (s) = r e^{-i \theta (s)}$, where $s = r \theta (s)$, $ds = r d \theta$,
$d \sigma = -i \sigma d \theta = -i \sigma R ds$.  Then,
\[
\overline{g(r)} = \int_{\Gamma_R} \!\!f( \sigma(s) ) \mu ( \sigma(s) , r )\,ds {-} \frac{R}{2 \pi} 
\int_{\Gamma_r}\!\! f( \sigma(s) )\,ds =
\int_{\Gamma_R}\!\! f( \sigma(s) ) \mu ( \sigma(s) , r )\,ds {+} \frac{1}{2 \pi i} \int_{\Gamma_r}
\!\!\frac{f( \sigma )}{\sigma}\,d \sigma .
\]
On $\Gamma_R$, writing $\sigma (s) = R e^{-i \theta (s)}$, where $s = R \theta (s)$, $ds = R d \theta$,
$d \sigma = i \sigma d \theta = ( i \sigma / R) ds$, and
using the fact that $\int_{\partial {\mathcal A}_R} \frac{f( \sigma )}{\sigma}\,d \sigma = 0$,
we obtain
\[
\overline{g(r)} = \int_{\Gamma_R}\!\! f( \sigma(s) ) \mu ( \sigma(s) , r )\,ds - \frac{1}{2 \pi i} \int_{\Gamma_R}\!\!
\frac{f( \sigma )}{\sigma}\,d \sigma =
\int_{\Gamma_R}\!\! f( \sigma(s) ) \left( \mu ( \sigma(s) , r ) - \frac{1}{2 \pi R} \right)\,ds .
\]
Finally, we observe
that if $\sigma = R e^{i \theta}$, then $\mu ( \sigma , r ) - \frac{1}{2 \pi R} =
\frac{1}{2 \pi R} \frac{R^4 - 1}{R^4 - 2 R^2 \cos \theta + 1} > 0$, which implies
\[
| g(r) | \leq \int_{\Gamma_R} \left( \mu ( \sigma(s) , r ) - \frac{1}{2 \pi R} \right)\,ds = 1 .
\]
Similarly, for proving that $|g(z){-}\gamma (f)|\leq \frac{2}{R^2{+}1}$ for all $z\in\Aa$ it suffices to show that 
$|g(r){-}\gamma (f)|\leq \frac{2}{R^2{+}1}$. For that, we observe that
\[
\gamma (f)=\frac{1}{2\pi }\frac{R^2{-}1}{R^2{+}1}\int_{0}^{2\pi }\,f(e^{i\theta })\,d\theta=
\frac{1}{2\pi R}\frac{R^2{-}1}{R^2{+}1}\int_{\Gamma _R }f(\sigma (s))\,ds,
\]
since $f$ is holomorphic in $\Aa$; therefore,
\begin{align*}
\overline{g(r)}{-}\gamma (f) &=\int_{\Gamma_R}\!\! f( \sigma(s) ) \Big(\mu ( \sigma(s) , r ){-} \frac{1}{2\pi R}{-}\frac{1}{2\pi R}\frac{R^2{-}1}{R^2{+}1}\Big)\,ds  \\
&=\frac{1}{\pi }\frac{R^2{-}1}{R^2{+}1}\int_0^{2\pi }f(R\,e^{i\theta })\frac{R^2(1+\cos\theta )}{R^4 {-} 2 R^2 \cos \theta {+} 1}\,d\theta,\quad\textrm{whence}
\end{align*}
\begin{align*}
|\overline{g(r)}{-}\gamma (f)| &\leq
\frac{1}{\pi }\frac{R^2{-}1}{R^2{+}1}\int_0^{2\pi }\frac{R^2(1+\cos\theta)}{R^4 {-} 2 R^2 \cos \theta {+} 1}\,d\theta =\frac{2}{R^2{+}1}.
\end{align*}

\end{proof}

\section{Estimates for the double layer potential kernel.}
We first assume that we are in the quantum case $\|A\|<R$ and $\|A^{-1}\|<R$. Then, we introduce
 the self-adjoint matrix $\nu(\sigma ,A)=\mu (\sigma ,A){-}\frac1{2\pi i}\frac{\sigma '}{\sigma }I$. If $\sigma \in\Gamma _{R}$, we may write $\sigma =Re^{i\theta }$, 
 $s=R\,\theta $; thus,
\begin{align*}
2\pi R\,\nu(\sigma ,A)&=Re^{i\theta} (Re^{i\theta}I{-}A)^{-1}{+}Re^{-i\theta} (Re^{-i\theta}I{-}A^*)^{-1}-I\\&=(Re^{i\theta}I{-}A)^{-1}(R^2I{-}AA^*)(Re^{-i\theta}I{-}A^*)^{-1}\geq 0,
\end{align*}
and if $\sigma=re^{-i\theta } \in\Gamma _r$, with $B=A^{-1}$,
\begin{align*}
2\pi r\,\nu(\sigma ,A)&=-re^{-i\theta} (re^{-i\theta}I{-}A)^{-1}{-}re^{i\theta} (re^{i\theta}I{-}A^*)^{-1}+I
\\&=(re^{-i\theta}I{-}A)^{-1}(AA^*{-}r^2I)(re^{i\theta}I{-}A^*)^{-1},\\
&=r^2(re^{-i\theta}I{-}A)^{-1}A(R^2{-}BB^*)A^*(re^{i\theta}I{-}A^*)^{-1}\geq 0.
\end{align*}
 Therefore, we have $\nu(\sigma ,A)\geq 0$ for all $\sigma \in\partial \Aa$.
Now, we notice that $\int_{\aR }(\mu (\sigma ,A){-}\nu(\sigma ,A)\,ds=\frac{1}{2\pi i}\int_{\partial \Aa}\frac{d\sigma }\sigma \, I=0$;  thus, if $f$ 
is a rational function bounded by 1 in $\Aa$, we obtain
\[
\|S(f,A)\|=\Big\|\int_{\partial \omega}f(\sigma )\nu (\sigma(s) ,A)\,ds\Big\|
\leq  \Big\|\int_{\aR}\nu(\sigma(s) ,A)\,ds\Big\|=\Big\|\int_{\aR}\mu(\sigma(s) ,A)\,ds\Big\|=2,
\]
where we have used the fact that $S(1,A) = 2I$ from (\ref{S1zA}). This shows that
\begin{align*}
\|S(f,A)\|\leq 2, \textrm {  for all matrices $A$ with $\|A\|<R$ and } \|A^{-1}\|<R. 
\end{align*}
Therefore, using \eqref{eq10} with $c_1=0$ and $c_2=\gamma (f)$, we get $a\leq 1$ and $b\leq \frac2{R+1}$; thus,
\[
K(A,R)\leq 1+\sqrt{1+\frac2{R+1}}.
\]
\bigskip

We now consider the numerical case $w(A)< R$ and $w(A^{-1})<R$, and we set
$\nu(\sigma ,A)=\mu (\sigma ,A)$ if $\sigma \in\Gamma _R$ and $\nu(\sigma ,A)=\mu (\sigma ,A){-}\frac1{\pi i}\frac{\sigma '}{\sigma }I$ if 
$\sigma \in\Gamma _r$;  then, we have $\nu(\sigma ,A)\geq 0$ on $\Gamma _R$ since  $w(A)\leq R$. 
On $\Gamma _r$,
\begin{align*}
2\pi r\,\nu(\sigma ,A)&=-re^{-i\theta} (re^{-i\theta}I{-}A)^{-1}{-}re^{i\theta} (re^{i\theta}I{-}A^*)^{-1}+2\,I
\\&=(re^{-i\theta}I{-}A)^{-1}(2AA^*{-}re^{i\theta}A{-}re^{-i\theta})(re^{i\theta }{-}A^*)^{-1},\\
&=(re^{-i\theta}I{-}A)^{-1}A\big(2I{-}r(e^{i\theta }B{+}e^{-i\theta }B^*)\big)A^*(re^{i\theta}I{-}A^*)^{-1}\geq 0,
\end{align*}
since $w(B)<R$. Hence, with $\gamma _1(f)=\frac1\pi \int_0^{2\pi }f(\sigma (re^{-i\theta }))\,d\theta $,
\begin{align*}
\|S(f,A){-}\gamma _1(f)\,I\|&=\Big\|\int_{\partial \omega}f(\sigma )\nu (\sigma(s) ,A)\,ds\Big\|
\leq  \Big\|\int_{\aR}\nu(\sigma(s) ,A)\,ds\Big\|\\
&\leq \Big\|\int_{\aR}\mu(\sigma(s) ,A)\,ds\Big\|+\frac1\pi \int_{0}^{2\pi }ds=4.
\end{align*}
Therefore, using \eqref{eq10} with $c_1=\gamma _1(f)$ and $c_2=\gamma (f)$, we get $a\leq 2$ and $b\leq \frac2{R+1}$;  
thus, in the numerical case, we obtain
\[
K(A,R)\leq 2+\sqrt{4+\frac2{R+1}}.
\]

\end{document}